\documentclass[11pt, onecolumn]{IEEEtran}

\usepackage{amssymb}
\usepackage{amsmath}
\usepackage{amsthm}
\usepackage{latexsym}
\usepackage{mathabx}
\usepackage[symbol]{footmisc}

\newtheorem{theorem}{Theorem}
\newtheorem{lemma}[theorem]{Lemma}
\theoremstyle{remark}
\newtheorem*{remark*}{Remark}

\title{The $p$-norm of circulant matrices \\via Fourier analysis}
\author{\IEEEauthorblockN{K.~R.~Sahasranand$^\ast$}}

\begin{document}
\maketitle

{\renewcommand{\thefootnote}{}\footnotetext{
\noindent$^\ast${Department of Electrical Communication Engineering, Indian Institute of Science.  Email: sahasranand@iisc.ac.in }}}

\renewcommand{\thefootnote}{\arabic{footnote}}
\setcounter{footnote}{0}
\vspace{-0.4cm}
\begin{abstract}
A recent paper~\cite{bouthat2021p} computed the induced $p$-norm of a special class of circulant matrices $A(n,a,b) \in \mathbb{R}^{n \times n}$, with the diagonal entries equal to $a \in \mathbb{R}$ and the off-diagonal entries equal to $b \ge 0$. We provide shorter proofs for all the results therein using Fourier analysis. The key observation is that a circulant matrix is diagonalized by a DFT matrix. We obtain an exact expression for $\|A\|_p, 1 \le p \le \infty$, where $A = A(n,a,b), a \ge 0$ and for $\|A\|_2$ where $A = A(n,-a,b), a \ge 0$; for the other $p$-norms of $A(n,-a,b)$, $2 < p < \infty$, we provide upper and lower bounds.
\end{abstract}
\vspace{-0.2cm}
\section{Introduction}
Circulant matrices arise in many applications ranging from wireless communications~\cite{tse2005fundamentals} to cryptography~\cite{pub2001announcing} to solving differential equations~\cite{wilde1983differential} (see~\cite{bouthat2021p} and the references therein for the historical context and more recent theoretical studies on circulant matrices). A circulant matrix is of the form
\[
A = \left[
  \begin{array}{ccccc}
    a_1 & a_2   & a_3 & \ldots & a_n    \\
    a_n & a_1   & a_2 & \ldots & a_{n-1}    \\
    a_{n-1} & a_n & a_1 & \ldots & a_{n-2}\\
	\vdots & \vdots & \vdots & \ddots & \vdots \\
	a_2 & a_3 & a_4 & \ldots & a_1
  \end{array}
\right],
\]
where $a_j \in \mathbb{R}, 1 \le j \le n$. For a matrix $A \in \mathbb{R}^{n \times n}$, we define the operator norm
\[
\|A\|_p = \sup_{x \neq 0} \frac{\|Ax\|_p}{\|x\|_p},
\]
for $1 \le p \le \infty$, where, for a vector $y = (y_1,\ldots,y_n) \in \mathbb{R}^n$,
\[
\|y\|_\infty = \max\left\{|y_1|,\ldots,|y_n|\right\},
\]
and for $1 \le p < \infty$,
\[
\|y\|_p = \left(|y_1|^p + \cdots + |y_n|^p\right)^{1/p}.
\]
It is well-known~\cite{rmgraycirculant} that the eigen decomposition of a circulant matrix $A$ is of the form $F^* \Lambda F$ where $F$ denotes the Discrete Fourier Transform (DFT) matrix (proof provided in the appendix for completeness); $F_{jk} = \frac{1}{\sqrt{n}}\cdot \omega_n^{-jk}, 0 \le j,k \le n-1$, where $\omega_n = e^{\frac{2\pi i}{n}}$ and for a matrix $B$, $B_{jk}$ denotes its $(j,k)$-th entry and $B^*$ denotes its adjoint. The eigenvalues, namely the diagonal entries of $\Lambda$, are given by
\begin{equation}
\lambda_k := \Lambda_{kk} = \sum_{j=0}^{n-1} a_{j+1} \omega_n^{jk}, ~~0 \le k \le n-1.
\label{eqn:eval}
\end{equation}
We use this property of circulant matrices to study the $p$-norm of a special class of circulant matrices, $A(n,a,b) \in \mathbb{R}^{n \times n}, a,b \in \mathbb{R}$, where
\[
A(n,a,b) := \left[
  \begin{array}{ccccc}
    a & b   & b & \ldots & b    \\
    b & a   & b & \ldots & b    \\
    b & b   & a & \ldots & b \\
	\vdots & \vdots & \vdots & \ddots & \vdots \\
	b & b & b & \ldots & a
  \end{array}
\right].
\]
The $1$-norm and the infinity norm of $A(n,a,b)$ are easily calculated to be $|a| + (n-1)|b|$ by inspection. As observed in~\cite{bouthat2021p}, it suffices to consider the following two cases: $A(n,a,b)$ and $A(n,-a,b)$ where $a,b \ge 0$. We obtain an exact expression for $\|A\|_p, 1 < p < \infty$, where $A = A(n,a,b)$ and for $\|A\|_2$ where $A = A(n,-a,b)$. For $A = A(n,-a,b)$, we provide upper and lower bounds for $\|A\|_p, 2 < p < \infty$.
\section{Results and Proofs}
For a diagonal matrix, all the induced $p$-norms are equal to the maximum of the absolute value of the entries~\cite{horn2012matrix}. We calculate this value for $\Lambda$.
\begin{lemma}
For $a,b \ge 0$, for $1 \le p \le \infty$,
\begin{enumerate}
\item[i.] for $A = A(n,a,b)$ and $A = F^*\Lambda F$, we have
\begin{align*}
\|\Lambda\|_p = a + (n-1)b,
\end{align*}
\item[ii.] for $A = A(n,-a,b)$ and $A = F^*\Lambda F$, we have
\[
\|\Lambda\|_p = \begin{cases}
-a + (n-1)b &\text{ if } 2a \le (n-2)b\\
a + b &\text{ otherwise.}
\end{cases}
\]
\end{enumerate}  
\label{l:lamnorm}
\end{lemma}
\begin{proof}
For $A = A(n,a,b), a,b \in \mathbb{R}$ and $A = F^*\Lambda F$, by~\eqref{eqn:eval}, the diagonal entries of $\Lambda$ are given by
\begin{align*}
\lambda_k &= a\omega_n^{kk} + \sum_{\underset{j \neq k}{j=0}}^{n-1} b\omega_n^{jk}\\
&= b \sum_{j=0}^{n-1}\omega_n^{jk} + (a-b)\omega_n^{kk},
\end{align*}
for $0 \le k \le n-1$. Using the well-known identity (see, for example,~\cite{stein2011fourier}),
\[
\sum_{j=0}^{n-1} \omega_n^{jk} = \begin{cases}
n &\text{ if } k = 0 ~(\hspace{-0.35cm}\mod n)\\
0 &\text{ otherwise,}
\end{cases}
\]
we have
\begin{align*}
\lambda_0 &= bn + (a-b) = a + (n-1)b,
\end{align*}
and for $0 < k \le n-1$,
\[
|\lambda_k| = |a-b|.
\]
The result follows by calculating $\|\Lambda\|_p = \underset{0\le k \le n-1}{\max} |\lambda_k|$ for $1 \le p \le \infty$, for $A(n,a,b)$ and $A(n,-a,b)$.
\end{proof}
We use Lemma~\ref{l:lamnorm} to derive an exact expression for $\|A\|_2$ for $A(n,a, b)$ as well as $A(n,-a,b)$, for $a,b \ge 0$.
\begin{theorem}
For $a,b \ge 0$, 
\begin{enumerate}
\item[i.] for $A = A(n,a,b)$, we have
\begin{align*}
\|A\|_2 = a + (n-1)b.
\end{align*}
\item[ii.] for $A = A(n,-a,b)$, we have
\[
\|A\|_2 = \begin{cases}
-a + (n-1)b &\text{ if } 2a \le (n-2)b\\
a + b &\text{ otherwise.}
\end{cases}
\]
\end{enumerate}  
\label{thm:2norm}
\end{theorem}
\begin{proof}
The result follows by observing that $\|A\|_2 = \|\Lambda\|_2$ since $F$ is unitary, and using Lemma~\ref{l:lamnorm}.
\end{proof}
\begin{remark*}
In~\cite{bani2008norm}, similar techniques are employed to calculate the (unitarily invariant) Schatten $p$-norms of block circulant matrices. 
\end{remark*}
As observed in~\cite{bouthat2021p}, $A = A(n,a,b), a \in \mathbb{R}, b \ge 0$ is self-adjoint and hence $\|A\|_p = \|A\|_q$, for $p$ and $q$ satisfying $\frac{1}{p} + \frac{1}{q} = 1$ (see~\cite{horn2012matrix}). Thus, it suffices to focus on either $p \in (1,2]$ or $p \in [2,\infty)$. First, we compute the $p$-norm of $A = A(n,a,b)$ with $a,b \ge 0$ for $p \ge 2$.
\begin{theorem}
For $A = A(n,a,b),a,b \ge 0$, for $p \ge 2$,
\[
\|A\|_p = a + (n-1)b.
\]
\end{theorem}
\begin{proof}
Using the vector $x = [1,1,\ldots,1]^T$, where $[\cdot]^T$ denotes the transpose, we have 
\[
\|A\|_p \ge a + (n-1)b. 
\]
Next, observe that 
\[
\|A\|_\infty = a + (n-1)b = \|A\|_2, 
\]
where the last identity is by Theorem~\ref{thm:2norm}. By the Riesz-Thorin interpolation theorem~\cite[Theorem $2.1$]{stein2011functional} we have, for every $0 < \theta < 1$,
\[
\|A\|_{p_\theta} \le \|A\|_q^{1-\theta}\|A\|_r^{\theta},
\]
where $p_\theta,q,$ and $r$ satisfy
\begin{equation}
\frac{1}{p_\theta} = \frac{1-\theta}{q} + \frac{\theta}{r}.
\label{eqn:rt}
\end{equation}
Setting $p_\theta = p, q = 2$, and $r = \infty$ in~\eqref{eqn:rt} yields 
\[
\|A\|_p \le \|A\|_\infty = a + (n-1)b.
\]
\end{proof}
\begin{remark*}
Using similar arguments as above, one can calculate $\|A\|_p$ of a general circulant matrix $A$ with non-negative entries $a_1,\ldots,a_n$ to be $a_1 + \cdots + a_n$. 
\end{remark*}
Next, we estimate the $p$-norm of $A(n,-a,b)$ with $a,b \ge 0$ for $p \ge 2$.
\begin{theorem}
\label{thm:pnorm}
For $A = A(n,-a,b), a,b \ge 0$ and $A = F^* \Lambda F$, for $p \ge 2$, we have
\begin{align*}
- a + (n-1)b &\le \|A\|_p \le  n^{\frac{1}{2}-\frac{1}{p}} \cdot (-a + (n-1)b) &\text{ if } 2a \le (n-2)b,\\
a + b &\le \|A\|_p \le  n^{\frac{1}{2}-\frac{1}{p}} \cdot (a+b) &\text{ if } 2a \ge (n-2)b.
\end{align*}
\end{theorem}
\begin{proof}
For $p \ge 2$ and $x \neq 0$, we have
\begin{align*}
\|Ax\|_p &\le \|Ax\|_2\\
&\le \|F^*\|_2 \cdot \|\Lambda\|_2 \cdot \|\widehat{x}\|_2,
\end{align*}
where $\widehat{x}=Fx$ denotes the Fourier transform of $x$. By Plancherel's relation, we have $\|\widehat{x}\|_2 = \|x\|_2$ and $\|F^*\|_2 = 1$. Hence
\[
\|Ax\|_p \le \|\Lambda\|_2 \cdot \|x\|_2,
\]
whereby
\[
\|A\|_p \le n^{\frac{1}{2}-\frac{1}{p}} \cdot \|\Lambda\|_2,
\]
where we have used the inequality $\|x\|_r \le \|x\|_p \cdot n^{\frac{1}{r}-\frac{1}{p}}$ for $2=r \le p$. The upper bounds in the theorem now follow from Lemma~\ref{l:lamnorm}. To obtain the lower bounds, we exhibit a vector $x \neq 0$ such that $\|Ax\|_p/\|x\|_p$ equals the quantity in the desired lower bound.
\begin{enumerate}
\item[i.] $2a \le (n-2)b$: for $x = [1,1,\ldots,1]^T$,
\[
\frac{\|Ax\|_p}{\|x\|_p} = |-a + (n-1)b| \ge -a + (n-1)b.
\]
\item[ii.] $2a \ge (n-2)b$: for $x = [-1,1,0,\ldots,0]^T$,
\[
\frac{\|Ax\|_p}{\|x\|_p} = \left(\frac{|a+b|^p + |-a-b|^p}{2}\right)^{1/p} = a+b.
\]
\end{enumerate}
\end{proof}
Finally, we provide an improved upper bound for the $p$-norm of $A(n,-a,b)$ for $p > 2$, using the Riesz-Thorin interpolation theorem.
\begin{theorem}
For $A = A(n,-a,b), a,b \ge 0$, for every $p > 2$, we have,
\begin{align*}
\|A\|_p \le \|A\|_2^{\frac{2}{p}}\|A\|_\infty^{1-\frac{2}{p}}.
\end{align*}
\end{theorem}
\begin{proof}
Setting $q = 2$ and $r = \infty$ in~\eqref{eqn:rt} yields
\[
\frac{1}{p_\theta} = \frac{1-\theta}{2}.
\]
We choose $p_\theta = p$ to get
\[
\|A\|_{p} \le \|A\|_{2}^{1-\theta} \|A\|_\infty^{\theta}.
\] 
Since $\theta = 1-2/p$, we have
\[
\|A\|_{p} \le \|A\|_2^{\frac{2}{p}}\|A\|_\infty^{1-\frac{2}{p}}.
\]
\end{proof}
In fact, using similar arguments we can show the following.
\begin{theorem}
For $A = A(n,-a,b), a,b, \ge 0$, for $p \ge 2$, $\|A\|_{p}$ is monotonically non-decreasing in $p$.
\end{theorem}
\begin{proof}
Fix $\beta > 0$. Setting $q = p-\alpha, \alpha > 0, r = p+\beta$, and $p_\theta = p$ in~\eqref{eqn:rt} yields
\[
\|A\|_{p} \le \|A\|_{p-\alpha}^{1-\theta}\|A\|_{p+\beta}^\theta.
\]
Choose $\alpha$ such that 
\[
\frac{1}{p-\alpha} + \frac{1}{p+\beta} = 1.
\]
Since $A$ is self-adjoint, we have $\|A\|_{p-\alpha} = \|A\|_{p+\beta}$, and hence
\[
\|A\|_p \le \|A\|_{p+\beta}.
\]
\end{proof}
\begin{remark*}
As a corollary, we get $\|A\|_p \ge \|A\|_2$ for all $p \ge 2$, which is the same as the lower bound in Theorem~\ref{thm:pnorm}.
\end{remark*}

\section{Summary}
Using the observation that a circulant matrix is diagonalized by a DFT matrix, we have computed the $p$-norm of a special class of circulant matrices $A(n,a,b) \in \mathbb{R}^{n \times n}$, with the diagonal entries equal to $a \in \mathbb{R}$ and the off-diagonal entries equal to $b \ge 0$. The $1$-norm and the infinity norm of $A(n,a,b)$ are easily calculated to be $|a| + (n-1)|b|$ by inspection. For $A = A(n,a,b)$ with $a,b \ge 0$, we show that for all $1 \le p \le \infty$,
\[
\|A\|_p = a + (n-1)b.
\]
For $A = A(n,-a,b)$ with $a,b \ge 0$, we obtain an exact expression for $\|A\|_2$. Since $A$ is self-adjoint, $\|A\|_q = \|A\|_p$ for conjugate pairs, $p$ and $q$. This, along with the Riesz-Thorin interpolation theorem, implies that for $2 \le p \le \infty$, $\|A\|_p$ is monotonically non-decreasing in $p$. Further, we show that for $2 \le p \le \infty$,
\[
\|A\|_2 \le \|A\|_p \le \|A\|_2^{\frac{2}{p}}\|A\|_\infty^{1-\frac{2}{p}}.
\]
An exact expression for $\|A\|_p, 2 < p < \infty$ for $A = A(n,-a,b),a,b\ge 0$, remains elusive.
\section*{Acknowledgments}
The author thanks Prof. Manjunath Krishnapur and Prof. Apoorva Khare for useful comments.

\begin{appendix}
We prove that for any circulant matrix $A$, the eigen decomposition is of the form $A = F^* \Lambda F$, where
\[
F^* = \frac{1}{\sqrt{n}} \left[
  \begin{array}{ccccc}
    1 & 1   & 1 & \ldots & 1    \\
    1 & \omega   & \omega^2 & \ldots & \omega^{n-1}    \\
    1 & \omega^2 & \omega^4 & \ldots & \omega^{2(n-1)}\\
	\vdots & \vdots & \vdots & \ddots & \vdots \\
	1 & \omega^{n-1} & \omega^{2(n-1)} & \ldots & \omega^{(n-1)(n-1)}
  \end{array}
\right],
\]
and $\omega = e^{\frac{2\pi i}{n}}$ (we drop the subscript $n$ of $\omega_n$ for brevity). Define a permutation matrix
\[
P := \left[
  \begin{array}{ccccc}
    0 & 1   & 0 & \ldots & 0    \\
    0 & 0   & 1 & \ldots & 0    \\
    0 & 0   & 0 & \ldots & 0	\\
	\vdots & \vdots & \vdots & \ddots & \vdots \\
	1 & 0 & 0 & \ldots & 0
  \end{array}
\right].
\]
\textbf{\emph{Claim.}} $P = F^* \Omega F$ where $\Omega$ is a diagonal matrix with entries $\Omega_{kk} = \omega^k, 0 \le k \le n-1$.
\begin{proof}
For $0 \le j,k \le n-1$, observe that
\begin{align*}
j\text{-th row of } F^* &= \frac{1}{\sqrt{n}}\left[\omega^{0j}~~\omega^{1j}~~\cdots~~\omega^{(n-1)j}\right],\\
j\text{-th row of } F^*\Omega &= \frac{1}{\sqrt{n}}\left[\omega^{0(j+1)}~~\omega^{1(j+1)}~~\cdots~~\omega^{(n-1)(j+1)}\right],\\
k\text{-th column of } F &= \frac{1}{\sqrt{n}}\left[\omega^{-0k}~~\omega^{-1k}~~\cdots~~\omega^{-(n-1)k}\right]^T.
\end{align*}
Thus,
\begin{align*}
(j,k)\text{-th element of } F^*\Omega F &= \frac{1}{n} \sum_{r=0}^{n-1} \omega^{r(j+1)}\omega^{-rk} = \begin{cases}
1 \text{ if } j = k-1~(\hspace{-0.35cm}\mod n)\\
0 \text{ otherwise.}
\end{cases}
\end{align*}
\end{proof}
Therefore, a circulant matrix
\begin{align*}
A &= \left[
  \begin{array}{ccccc}
    a_1 & a_2   & a_3 & \ldots & a_n    \\
    a_n & a_1   & a_2 & \ldots & a_{n-1}    \\
     a_{n-1} & a_n & a_1 & \ldots & a_{n-2}\\
	\vdots & \vdots & \vdots & \ddots & \vdots \\
	a_2 & a_3 & a_4 & \ldots & a_1
  \end{array}
\right]\\
&= a_1 I + a_2 P + a_3 P^2 + \cdots + a_n P^{n-1}\\
&= F^* \left(a_1 I + a_2 \Omega + a_3 \Omega^2 + \cdots + a_n \Omega^{n-1}\right)F\\
&= F^* \Lambda F.
\end{align*}
Observe that $\lambda_k = \Lambda_{kk}, 0 \le k \le n-1$, satisfy~\eqref{eqn:eval}.\\
\end{appendix}

\footnotesize \emph{Acknowledgement.} The author learnt this proof from Ashok Krishnan, IISc, at a Digital Communication study group meeting in $2011$.
\vspace{-.1cm}
\bibliography{references}

\begin{thebibliography}{1}
\providecommand{\url}[1]{#1}
\csname url@samestyle\endcsname
\providecommand{\newblock}{\relax}
\providecommand{\bibinfo}[2]{#2}
\providecommand{\BIBentrySTDinterwordspacing}{\spaceskip=0pt\relax}
\providecommand{\BIBentryALTinterwordstretchfactor}{4}
\providecommand{\BIBentryALTinterwordspacing}{\spaceskip=\fontdimen2\font plus
\BIBentryALTinterwordstretchfactor\fontdimen3\font minus
  \fontdimen4\font\relax}
\providecommand{\BIBforeignlanguage}[2]{{%
\expandafter\ifx\csname l@#1\endcsname\relax
\typeout{** WARNING: IEEEtran.bst: No hyphenation pattern has been}%
\typeout{** loaded for the language `#1'. Using the pattern for}%
\typeout{** the default language instead.}%
\else
\language=\csname l@#1\endcsname
\fi
#2}}
\providecommand{\BIBdecl}{\relax}
\BIBdecl

\bibitem{bouthat2021p}
\BIBentryALTinterwordspacing
L.~Bouthat, A.~Khare, J.~Mashreghi, and F.~Morneau-Gu{\'e}rin, ``The $p$-norm
  of circulant matrices,'' \emph{Linear and Multilinear Algebra}, pp. 1--13,
  2021. [Online]. Available:
  \url{https://doi.org/10.1080/03081087.2021.1983513}; \url{https://arxiv.org/abs/2109.09728}
\BIBentrySTDinterwordspacing

\bibitem{tse2005fundamentals}
D.~Tse and P.~Viswanath, \emph{Fundamentals of wireless communication}.\hskip
  1em plus 0.5em minus 0.4em\relax Cambridge university press, 2005.

\bibitem{pub2001announcing}
N.~F. Pub, ``Announcing the {A}dvanced {E}ncryption {S}tandard ({AES}),''
  \emph{Federal Information Processing Standards Publication}, vol. 197, pp.
  1--51, 2001.

\bibitem{wilde1983differential}
A.~C. Wilde, ``Differential equations involving circulant matrices,'' \emph{The
  Rocky Mountain Journal of Mathematics}, vol.~13, no.~1, pp. 1--13, 1983.

\bibitem{rmgraycirculant}
\BIBentryALTinterwordspacing
R.~M. Gray, ``Toeplitz and circulant matrices: A review,'' \emph{Foundations
  and Trends® in Communications and Information Theory}, vol.~2, no.~3, pp.
  155--239, 2006. [Online]. Available:
  \url{http://dx.doi.org/10.1561/0100000006}
\BIBentrySTDinterwordspacing

\bibitem{horn2012matrix}
R.~A. Horn and C.~R. Johnson, \emph{Matrix analysis}.\hskip 1em plus 0.5em
  minus 0.4em\relax Cambridge university press, 2012.

\bibitem{stein2011fourier}
E.~M. Stein and R.~Shakarchi, \emph{Fourier analysis: an introduction}.\hskip
  1em plus 0.5em minus 0.4em\relax Princeton University Press, 2011, vol.~1.

\bibitem{bani2008norm}
W.~Bani-Domi and F.~Kittaneh, ``Norm equalities and inequalities for operator
  matrices,'' \emph{Linear Algebra and Its Applications}, vol. 429, no.~1, pp.
  57--67, 2008.

\bibitem{stein2011functional}
E.~M. Stein and R.~Shakarchi, \emph{Functional analysis: Introduction to
  Further Topics in Analysis}.\hskip 1em plus 0.5em minus 0.4em\relax Princeton
  University Press, 2011, vol.~4.

\end{thebibliography}
\bibliographystyle{IEEEtran}
\end{document}